\documentclass{amsart}
\usepackage{amsmath,amscd,amsthm,amsfonts,amssymb}

\theoremstyle{plain}
\newtheorem{theorem}                 {Theorem}      [section]
\newtheorem{proposition}  [theorem]  {Proposition}

\theoremstyle{definition}
\newtheorem{example}      [theorem]  {Example}

\newtheorem{definition}   [theorem]  {Definition}

\numberwithin{equation}{section}

\def \theo-intro#1#2 {\vskip .25cm\noindent{\bf Theorem #1\ }{\it #2}}

\def \zn{\mathbb Z}

\def \rn{\mathbb R}
\def \cn{\mathbb C}
\def \hn{\mathbb H}

\def \can{\mathbb C\text{a}}

\def \proc#1{\cn  \text{P}^{#1}}
\def \proh#1{\hn  \text{P}^{#1}}
\def\proca#1{\can \text{P}^{#1}}

\def \g{\mathfrak{g}}
\def \h{\mathfrak{h}}
\def \k{\mathfrak{k}}
\def \m{\mathfrak{m}}

\def \p{\mathfrak{p}}

\def \t{\mathfrak{t}}

\def \SO#1{\text{\bf SO}(#1)}
\def \so#1{\mathfrak{so}(#1)}

\def \Spin#1{\text{\bf Spin}(#1)}

\def \U#1{\text{\bf U}(#1)}

\def \US#1{\text{\bf U}^*(#1)}

\def \SU#1{\text{\bf SU}(#1)}
\def \su#1{\mathfrak{su}(#1)}

\def \Sp#1{\text{\bf Sp}(#1)}

\def \nab#1#2{\hbox{$\nabla$\kern -.3em\lower 1.0 ex
    \hbox{$#1$}\kern -.1 em {$#2$}}}

\begin{document}
\baselineskip 22pt \larger

\allowdisplaybreaks

\title{Harmonic morphisms from homogeneous\\ spaces of positive curvature}

\author{Sigmundur Gudmundsson}
\author{Martin Svensson}

\keywords{harmonic morphisms, Riemannian homogeneous spaces,
minimal submanifolds}

\subjclass[2000]{58E20, 53C43, 53C12}

\address
{Department of Mathematics, Faculty of Science, Lund University,
Box 118, S-221 00 Lund, Sweden}
\email{Sigmundur.Gudmundsson@math.lu.se}

\address
{Department of Mathematics \& Computer Science, University of
Southern Denmark, Campusvej 55, DK-5230 Odense M, Denmark}
\email{svensson@imada.sdu.dk}

\begin{abstract}
We prove local existence of complex-valued harmonic
morphisms from any Riemannian homogeneous space of positive
curvature, except the Berger space $\Sp 2/\SU 2$.
\end{abstract}

\maketitle

\section{Introduction}

In Riemannian geometry, the notion of a minimal submanifold of a
given ambient space is of great importance.  Complex-valued harmonic
morphisms $\phi:(M,g)\to\cn$ are useful tools for the construction
of such objects. They are solutions to
an over-determined non-linear system of partial differential
equations determined by the geometric data of the manifold
involved. For this reason harmonic morphisms are difficult to find
and have no general existence theory.  There even exist manifolds
not admitting any solutions, not even locally.

The current authors have shown in \cite{Gud-Sve-4} that any
Riemannian symmetric space $(G/H,g)$, that is neither
$G_2/\SO 4$ nor its non-compact dual, carries local
complex-valued harmonic morphisms.  Even global solutions
exist if the symmetric space is of non-compact type.
It is natural to extend the investigation to the more general
Riemannian homogenous spaces.  Many homogeneous spaces fibre
$G/K\to G/H$ over a symmetric space and such Riemannian
fibrations are well behaved with respect to harmonic morphisms
as we prove in Theorem \ref{theo-homo}.  Applying this
fact we obtain our main result.

\begin{theorem}
Let $(M,g)$ be a Riemannian homogeneous space of positive
curvature, which is {\it not} the Berger space $\Sp 2/\SU 2$.
Then $M$ admits local complex-valued harmonic morphisms.
\end{theorem}

\section{Harmonic Morphisms}

Let $M$ and $N$ be two manifolds of dimension $m$ and $n$,
respectively. Then a Riemannian metric $g$ on $M$ gives rise
to the notion of a Laplacian on $(M,g)$ and real-valued harmonic
functions $f:(M,g)\to\rn$. This can be generalized to the concept
of a harmonic map $\phi:(M,g)\to (N,h)$ between Riemannian
manifolds being a solution to a semi-linear system of partial
differential equations.

\begin{definition}
A map $\phi:(M,g)\to (N,h)$ between Riemannian manifolds is
called a {\it harmonic morphism} if, for any harmonic function
$f:U\to\rn$ defined on an open subset $U$ of $N$ with
$\phi^{-1}(U)$ non-empty, the composition
$f\circ\phi:\phi^{-1}(U)\to\rn$ is a harmonic function.
\end{definition}

The following characterization of harmonic morphisms between
Riemannian manifolds is due to Fuglede and Ishihara,
see \cite{Fug-1},\cite{Ish}.

\begin{theorem}
A map $\phi:(M,g)\to (N,h)$ between Riemannian manifolds is a
harmonic morphism if and only if it is a horizontally (weakly)
conformal harmonic map.
\end{theorem}

The next result of Baird and Eells gives the theory of
harmonic morphisms a strong geometric flavour and shows that the
case when $n=2$ is particularly interesting. The conditions
characterizing harmonic morphisms are then independent of
conformal changes of the metric on the surface $N^2$. For the
definition of {\it horizontal homothety} we refer to \cite{Bai-Woo-book}.

\begin{theorem}\cite{Bai-Eel}\label{theo:B-E}
Let $\phi:(M^m,g)\to (N^n,h)$ be a horizontally (weakly) conformal
map between Riemannian manifolds. If
\begin{enumerate}
\item[i)] $n=2$, then $\phi$ is harmonic if and only if $\phi$ has
minimal fibres at regular points,
\item[ii)] $n\ge 3$, then two of the following
conditions imply the other
\begin{enumerate}
\item $\phi$ is a harmonic map,
\item $\phi$ has minimal fibres at regular points,
\item $\phi$ is horizontally homothetic.
\end{enumerate}
\end{enumerate}
\end{theorem}

The following useful existence result was proven in \cite{Gud-Sve-4}.

\begin{theorem}\label{theo-symmetric}
Let $(M,g)$ be an irreducible Riemannian symmetric space which is
neither $G_2/\SO 4$ nor its non-compact dual.  Then $M$ carries local
complex-valued harmonic morphisms.  If $M$ is of non-compact type
then global solutions exist.
\end{theorem}

For the general theory of harmonic morphisms, we refer to the
standard reference \cite{Bai-Woo-book} and the regularly updated
on-line bibliography \cite{Gud-bib}.

\section{The homogeneous Riemannian fibrations}

Let $G$ be a Lie group,  $K\subset H$ be compact subgroups of $G$
and $\k,\h,\g$ be their Lie algebras, respectively.  Then we have
the homogeneous fibration $$\pi:G/K\to G/H,\ \ \ \pi:aK\mapsto aH,$$
with fibres diffeomorphic to $H/K$.  Let $\m$ be an
$\text{Ad}(H)$-invariant complement of $\h$ in $\g$ and $\p$ be an
$\text{Ad}(K)$-invariant complement of $\k$ in $\h$.
Then $\p\oplus\m$ is an $\text{Ad}(K)$-invariant complement of $\k$ in $\g$.

Let $<,>_{\m}$ be an $\text{Ad}(H)$-invariant scalar product
on $\m$ inducing a $G$-invariant Riemannian metric $g$ on $G/H$.
Further let $<,>_{\p}$ be an $\text{Ad}(K)$-invariant scalar product
on $\p$ defining a $H$-invariant Riemannian metric $\bar g$ on $H/K$.
Then the orthogonal sum $$<,>=<,>_{\m}+<,>_{\p}$$ on $\m\oplus\p$
defines a $G$-invariant Riemannian metric $\hat g$ on $G/K$. For this
situation we have the following well-known result due to
B\' erard-Bergery, see \cite{B-B-1} or \cite{Bes}.

\begin{proposition}\label{prop:Bergery}
The homogeneous projection $\pi:(G/K,\hat g)\to (G/H,g)$ is a Riemannian
submersion with totally geodesic fibres.
\end{proposition}

As an immediate consequence we have the following useful result.

\begin{proposition}\label{prop-homo-lift}
Let $\pi:(G/K,\hat g)\to (G/H,g)$ be a homogeneous Riemannian
fibration. Further let $\phi:(G/H,g)\to (N,h)$ be a map to a
Riemannian manifold and $\hat\phi:(G/K,\hat g)\to(N,h)$ be the
composition $\hat\phi=\phi\circ\pi$.  Then $\phi$ is a harmonic
morphism if and only if $\hat\phi$ is a harmonic morphism.
\end{proposition}

\begin{proof}
The statement is an immediate consequence of Proposition 1.1 of
\cite{Gud-4} and the fact that by Theorem 2.3, every Riemannian
submersion with totally geodesic fibres is a harmonic morphism.
\end{proof}

\begin{theorem}\label{theo-homo}
Let $\pi:(G/K,\hat g)\to (G/H,g)$ be a homogeneous Riemannian
fibration.  If the base $G/H$ is a symmetric space, neither $G_2/\SO 4$
nor its non-compact dual, then the homogeneous space $G/K$ admits
local complex-valued harmonic morphisms.  If $G/H$ is of non-compact
type then global solutions exist.
\end{theorem}

\begin{proof}
The statement is a direct consequence of Proposition \ref{prop-homo-lift}
and Theorem \ref{theo-symmetric}.
\end{proof}

\section{The homogeneous spaces of positive curvature}

A simply connected normal Riemannian homogeneous space of
positive sectional curvature is either a symmetric space $S^n$,
$\proc n$, $\proh n$, $\proca 2$ of rank 1 or one of the following
$$\SU{n+1}/\SU n,\ \Sp{n+1}/\Sp n,\ \Spin 9/\Spin 7,$$
$$\Sp{n+1}/\Sp n\times S^1,\ \Sp 2/\SU 2,$$
$$\SU 5/\Sp 2\times S^1,\ (\SU 3\times\SO 3)/\U 2^*\cong W^7_{1,1}.$$
The first four displayed spaces are not symmetric but diffeomorphic to
$S^{2n+1}$, $S^{4n+3}$, $S^{15}$ and $\proc {2n+1}$, respectively.
Berger gave a classification of the diffeomorphism types in \cite{Ber},
but missed out the last space on the list.  This was later added by
Wilking in \cite{Wil}.

In \cite{Wal}, Wallach dropped the condition of normality and
gave a general classification for the even dimensional case.
With this he added to the list by constructing interesting
families of invariant metrics on the homogeneous spaces
$$W^6=\SU 3/{\bf S}(\U 1\times\U 1\times\U 1),$$
$$W^{12}=\Sp 3/(\Sp 1\times\Sp 1\times\Sp 1),$$
$$W^{24}=F_4/\Spin 8.$$

A classification for the general odd dimensional case was given
by B\' erard-Bergery in \cite{B-B-2}.  He showed that the list had
been completed earlier by Aloff and Wallach in \cite{Alo-Wal}.
They added the following infinite family of 7-dimensional examples
$$W^7_{k,l}=\SU 3/T^1_{k,l}\ \ \text{where}
\ \ k,l\in\zn\ \ \text{and}\ \ \text{gcd}(k,l)=1.$$

\section{The normal spaces}

The normal homogeneous space $\SU{n+1}/\SU n$, $\Sp{n+1}/\Sp n$,
$\Sp{n+1}/\Sp n\times S^1$ and $\Spin 9/\Spin 7$ fibre over the
symmetric $\proc n$, $\proh n$ and $S^8$ via the inclusions
$$\SU n\hookrightarrow{\bf S}(\U n\times\U 1)\hookrightarrow\SU{n+1},$$
$$\Sp n\hookrightarrow\Sp n\times\Sp 1\hookrightarrow\Sp{n+1},$$
$$\Sp n\times S^1\hookrightarrow\Sp n\times\Sp 1\hookrightarrow\Sp{n+1},$$
$$\Spin 7\hookrightarrow\Spin 8\hookrightarrow \Spin 9.$$
A good reference for this is Wallach's paper \cite{Wal}.
The embedding $\SU 2\hookrightarrow\Sp 2$ is {\it not} the standard
one that can be found in Cartan's classification for symmetric spaces.
Our methods do {\it not} work for this case since $\SU 2$ sits in
$\Sp 2$ as a maximal compact proper subgroup. For the details we
recommend \cite{Wil} and \cite{Eli} where the sectional curvatures
are determined.

\begin{example}
Let the special unitary group $\SU 5$ be equipped with its
standard bi-invariant Riemannian metric induced by the Killing
form of its Lie algebra $\su 5$.  Let
$\Sp 2\times_{\zn_2}\U 1$ be the subgroup of $\SU 5$ given by
$$\{\begin{pmatrix} zA & 0  \\ 0 & \bar z^4 \end{pmatrix}|\
A\in\Sp 2, z\in\U 1\}.$$  Then we have the sequence of group
monomorphisms
$$\Sp 2\times_{\zn_2}\U 1\hookrightarrow
{\bf S}(\U 4\times\U 1)\hookrightarrow\SU 5$$
inducing the homogeneous Riemannian fibration
$$\SU 5/(\Sp 2\times_{\zn_2}\U 1)\to\SU 5/{\bf S}(\U 4\times\U 1).$$
This is a harmonic morphism over the symmetric
complex projective space $\proc 4$.
\end{example}

For Example \ref{exam-last} we need the following general discussion.
Let $A,B$ be Riemannian Lie groups and $\alpha:C\to A$, $\beta:C\to B$ be
group homomorphisms from the Lie group $C$.  Let $C^*$ be the
subgroup of $\alpha(C)\times\beta(C)$ which is the image
of $(\alpha, \beta):C\to A\times B$.  Then the map
$$\phi:(A\times B)/C^*\to(A\times B)/(\alpha(C)\times\beta(C))$$
is a homogeneous Riemannian fibration and hence a
harmonic morphism. The epimorphism
$$\psi:(A\times B)/(\alpha(C)\times\beta(C))
=A/\alpha(C)\times B/\beta(C)\to A/\alpha(C)$$
is clearly a harmonic morphism, so the same applies to the composition
$$\psi\circ\phi:(A\times B)/C^*\to A/\alpha(C).$$

\begin{example}\label{exam-last}

We define a continuous family of left-invariant Riemannian metrics
on the Lie group $\SU 3\times\SO 3$.  With $t\in\rn^+$ this is given
by the scalar products
$$<,>_{t}=t<,>_{\su 3}+<,>_{\so 3}$$
on the Lie algebra $\su 3\oplus\so 3$, where $<,>_{\su 3}$ and
$<,>_{\so 3}$ are the negative Killing forms of $\su 3$ and
$\so 3$, respectively.  Let $C=\U 2$ be  the unitary group,
$\alpha:\U 2\to\SU 3$ be the group homomorphim with
$$\alpha(A)=\begin{pmatrix} A & 0  \\ 0 & \det A^{-1} \end{pmatrix}$$
and $\beta:\U 2\to\SU 2\to\SO 3$ be the canonical epimorphism.
Then
$$\alpha(\U 2)={\bf S}(\U 2\times \U 1),\ \ \beta(\U 2)=\SO 3$$
and $\U 2^*$ is the subgroup of $\alpha(\U 2)\times\beta(\U 2)$
which is the image of $$(\alpha, \beta):\U 2\to\SU 3\times\SO 3.$$
The above construction gives a homogeneous Riemannian fibration
$$(\SU 3\times\SO 3)/\U 2^*\to\SU 3/\text{\bf S}(\U 2\times\U 1).$$
Here the metric on the symmetric space
$\cn P^2=\SU 3/\text{\bf S}(\U 2\times\U 1)$ is a
constant multiple of the standard Fubini-Study metric.
\end{example}

\section{The Wallach spaces}

It follows from Wallach's construction of the spaces $W^6$,
$W^{12}$, $W^{24}$ that they fibre over the symmetric spaces
$\proc 2$, $\proh 2$ and $\proca 2$, respectively.  We shall here
describe the complex case and refer the reader to the original
paper \cite{Wal} for the rest. We also recommend Ziller's
excellent survey \cite{Zil}.

The Lie algebra $\k$ of the maximal torus
$K={\bf S}(\U 1\times\U 1\times\U 1)$ of the special unitary group
$\SU 3$ is given by
$$\k=\{\begin{pmatrix} \alpha & 0 & 0
\\ 0 & \beta & 0 \\ 0 & 0 & \gamma
\end{pmatrix}\in \su 3|\ \alpha+\beta+\gamma=0\}.$$
We have a natural splitting $\su 3=\k\oplus\p\oplus\m$, where
$$\p=\{\begin{pmatrix} 0 & 0 & 0 \\ 0 & 0
& z \\ 0 & -\bar z & 0\end{pmatrix}\in\su 3|\ z\in\cn\}$$
and
$$\m=\{\begin{pmatrix} 0 & z & w \\ -\bar z & 0
& 0 \\ -\bar w & 0 & 0\end{pmatrix}\in\su 3|\ z,w\in\cn\}.$$
Then $\h=\k\oplus\p$ is the Lie algebra of the subgroup
$$H={\bf S}(\U 1\times \U 2)
=\{\begin{pmatrix} \det A^{-1} & 0  \\ 0 & A
\end{pmatrix}\in\SU 3|\ A\in\U 2\}$$
and $\p\oplus\m$ is an $\text{Ad}(K)$-invariant complement of $\k$.
With $t\in\rn^+$ we define a family of
Euclidean scalar products on $\p\oplus\m$ with
$$<,>_{t}=t<,>_{\p}+<,>_{\m},$$ where $<,>_{\p}$
and $<,>_{\m}$ are the negative Killing forms
restricted to $\m$ and $\p$, respectively.  We now have the
inclusions
$${\bf S}(\U 1\times\U 1\times\U 1)
\hookrightarrow{\bf S}(\U 1\times\U 2)
\hookrightarrow\SU 3.$$
This provides us with a family of homogeneous Riemannian fibrations
$$\SU 3/{\bf S}(\U 1\times\U 1\times\U 1)
\to\SU 3/{\bf S}(\U 1\times\U 2)=\proc 2$$
over the symmetric $\proc 2$.  The Wallach spaces $W^6$ of positive
curvature are all contained in this family.

\section{The Aloff-Wallach spaces}

For elements $k,l\in\zn$ with $k^2+l^2\neq 0$ let
$T^1_{k,l}$ be the circle subgroup of $\SU 3$ given by
$$T^1_{k,l}=\{\begin{pmatrix} z^k & 0 & 0 \\ 0 & z^l
& 0 \\ 0 & 0 & \bar z^{k+l}\end{pmatrix}\in\SU 3|\ z\in\U 1\}.$$
Then $T^1_{k,l}$ is clearly also a subgroup of
$$H={\bf S}(\U 1\times \U 2)=\{\begin{pmatrix} \det A^{-1} & 0  \\ 0 & A
\end{pmatrix}\in\SU 3|\ A\in\U 2\}.$$

We have a decomposition $\su 3=\t_{k,l}\oplus\p\oplus\m$
orthogonal with respect to the Killing form of $\su 3$
where $\t_{k,l}$ and $\h=\t_{k,l}\oplus\p$ are the Lie
algebras of $T^1_{k,l}$ and $H$, respectively.  With $t\in\rn^+$
we define a family of $\text{Ad}(K)$-invariant scalar products on
$\p\oplus\m$ with $$<,>_{t}=t<,>_{\p}+<,>_{\m},$$ where $<,>_{\p}$
and $<,>_{\m}$ are the negative Killing form
restricted to $\m$ and $\p$, respectively. Now the inclusions
$$T^1_{k,l}\hookrightarrow
{\bf S}(\U 1\times \U 2)
\hookrightarrow\SU 3$$
give a family of homogeneous Riemannian fibrations
$$W^7_{k,l}=\SU 3/T^1_{k,l}\to\SU 3/{\bf S}(\U 1\times \U 2)=\proc 2$$
over the symmetric complex projective plane.  The homogeneous
Aloff-Wallach spaces of positive curvature are characterized
by $\text{gcd}(k,l)=1$ and they are all contained
in the family described above, see \cite{Alo-Wal}.
Wilking has pointed out that  $W^7_{1,1}\cong(\SU 3\times\SO 3)/\US 2$.

\section{Acknowledgements}
The second author was supported by the Danish Council for Independent Research under the project {\it Symmetry Techniques in Differential Geometry}.

\end{document}